\documentclass{aims} 
\usepackage{amsmath}
\usepackage{amssymb}
\usepackage{paralist}
\usepackage[misc]{ifsym}
\usepackage{epsfig} 
\usepackage{epstopdf} 
\usepackage[colorlinks=true]{hyperref}
\hypersetup{urlcolor=blue, citecolor=red}
\allowdisplaybreaks

\def\currentvolume{X}
 \def\currentissue{X}
  \def\currentyear{200X}
   \def\currentmonth{XX}
    \def\ppages{X-XX}
     \def\DOI{10.3934/xx.xxxxxxx}

\newtheorem{theorem}{Theorem}[section]
\newtheorem{corollary}[theorem]{Corollary}

\newtheorem{lemma}[theorem]{Lemma}
\newtheorem{proposition}[theorem]{Proposition}

\theoremstyle{definition}

\newtheorem{remark}[theorem]{Remark}

\newcommand{\ainf}{a^\infty}
\newcommand{\Binf}{B^\infty}
\newcommand{\pinf}{p^\infty}
\newcommand{\rinf}{r^\infty}
\newcommand{\uinf}{u^\infty}
\newcommand{\usim}{u^\sim}
\newcommand{\peff}{p^{\text{eff}}}
\newcommand{\decayfunc}{\Psi}
\newcommand{\rrr}{\mathfrak{r}}
\newcommand{\sss}{\mathfrak{s}}
\newcommand{\altil}{\tilde{\alpha}}
\newcommand{\kernel}{g}
\newcommand{\revision}[1]{{ #1}}

\newcommand{\Margin}[1]{}
\newcommand{\rerevision}[1]{{ #1}}
\newcommand{\rerevisionown}[1]{{ #1}}
\newcommand{\MMargin}[1]{}

\title[subdiffusion with counteracting terms]
{Decay rates to equilibrium in a nonlinear subdiffusion equation with two counteracting terms
} 

\author[Barbara Kaltenbacher]{}


\thanks{This research was funded in part by the Austrian Science Fund (FWF) 
[10.55776/F100800]. 
}

\thanks{$^*$Corresponding author: Barbara Kaltebacher}

\begin{document}
\maketitle

\centerline{\scshape
Barbara Kaltenbacher$^{{\href{mailto:barbara.kaltenbacher@aau.at}{\textrm{\Letter}}}*1}$
}

\medskip

{\footnotesize
 \centerline{$^1$Department of Mathematics,
Alpen-Adria-Universit\"at Klagenfurt, Austria}
} 

\medskip


\bigskip

 \centerline{(Communicated by Handling Editor)}


\begin{abstract}
In this paper we prove convergence to a steady state as $t\to\infty$ for solutions to the subdiffusion equation 
\[
\partial_t^\alpha u - \mathbb{L} u = q(x)u - p(x)f(u) + r
\]
with the exponential ($\alpha=1$) or power law ($\alpha\in[0,1)$) rates under mild conditions on the coefficients $p$, $q$, the nonlinearity $f$, the source $r$, and the elliptic operator $\mathbb{L}$.
\end{abstract}


\section{Introduction}
The model 
\[
u_t - \mathbb{L} u = q(x)u - p(x)f(u) \revision{+r}
\]
\Margin{RevX 1.}
and its subdiffusion counterpart
\[
\partial_t^\alpha u - \mathbb{L} u = q(x)u - p(x)f(u) \revision{+r} 
\]
with an elliptic differential operator $\mathbb{L}$ and the Djirbashian-Caputo fractional derivative of order $\alpha\in(0,1)$ appears in a wide range of applications from population dynamics \cite{Fisher:1937,KPP} via chemical reactions \cite{FrankZeldovich} to phase separation \cite{AllenCahn:1972,PenroseFife:1990}.

Here $p$, $q$, 
and often also 
$f'$
are typically nonnegative, so that the two terms $q(x)u$ and  $p(x)f(u)$ play counteracting roles in the PDE.

While spatial inhomogeneity of the coefficients $p$ and $q$ arises quite naturally in the mentioned applications, the literature on the non-constant setting is quite scarce; this is probably due to the fact that most of the interest went into investigation of travelling wave solutions, which are largely inhibited by spatially varying coefficients $p$ and $q$, though.
The large time limit has already been studied in \cite{AronsonWeinberger} and steady state solutions of the Fisher-KPP equation as well as their relations to travelling wave solutions have been subject of intensive research, see, e.g., \cite{BerestyckiNadinPerthameRyzhik:2009,CabreRoquejoffre:2012} for the non-local setting. 
The same holds true for the steady state solutions of the Allen-Cahn equation, that can be characterized as minimizers of an appropriate energy functional, see, e.g., \cite{AkagiSchimpernaSegatti:2016,NinomiyaHirokazuTaniguchi:2005}; in particular, we point to \cite{DuYangZhou:2020}, \cite[Section 6.4.2]{Jin:2021} for the time-fractional case.

However, a quantification of convergence to steady state in the time fractional case 
\revision{of the setting with two competing terms still seems to be missing so far. 
}
\Margin{RevY (4)}
We point to the recent result \cite[Proposition 4.7.]{LuchkoYamamoto:2025} in case $p(x)\equiv1$, \rerevisionown{but with $q\leq0$, thus not counteracting the nonlinear term.}
Providing such an analysis (allowing also for spatially varying coefficients) is the aim of this paper.

We mention in passing that this work is also motivated by its need for the convergence proof of a fixed point scheme used to reconstruct $p(x)$ and $q(x)$ from indirect observations in \cite{BB19}, where also more details on the model aspect and related literature can be found.
\rerevisionown{For this purpose we also prove convergence to zero of the time derivative $u_t$.}

\medskip

Consider the initial boundary value problem
\begin{equation}\label{ibvp}
\begin{aligned}
\partial_t^\alpha u-\mathbb{L}u&=qu-pf(u)+r \text{ in }(0,T)\times\Omega, \\ 
B u&=a \text{ on }(0,T)\times\partial\Omega, \\
u&=u^0\text{ in }\{0\}\times\Omega
\end{aligned}
\end{equation}
where $\Omega\subseteq\mathbb{R}^d$, $d\in\{1,2,3\}$ is a smooth and bounded domain and the operator \revision{$B$ 
is either the Dirichlet or the Neumann trace.
Moreover, the space and time dependent function $r(t,x)$ models the presence of a source or driving term.}
\Margin{RevX 1.}
\Margin{RevX 3.}

Without loss of generality we assume 
\begin{equation}\label{f0fprime0}
f(0)=0,\qquad f'(0)=0
\end{equation}
as this can be achieved without changing the problem by making the replacements 
\rerevision{
\[
f\to f(u)-f'(0)u-f(0), \quad q\to q+f'(0)p, \quad r\to r-f(0)p.
\]
In fact, this replacement includes the setting $f(\xi)=o(\xi)$ as $\xi\to0$ standardly used in parabolic problems, as long as $f$ is continuously differentiable, since a possible nonzero derivative $f'(0)$ can be shifted into the $q$ term.} 
\MMargin{RevY (1)}

We will require $p$ and $f'$ to be positive, but do not impose any sign conditions on $r$.
Essential sign (in fact, monotonicity) conditions that indirectly also affect $q$ will be imposed below, cf. \eqref{pinfty}, \eqref{boundpqf_int} below.

The elliptic operator is assumed to be of the form
\begin{equation}\label{ellipticityL}
\mathbb{L}=
\nabla\cdot(D\nabla)
\text{ with }
\lambda_{\min}(D(x))\geq \underline{c}_D>0 \quad \text{for a.e. }x\in\Omega
\end{equation}
\Margin{RevX 2.}
\Margin{RevY (1)}
satisfying (along with the boundary conditions) elliptic regularity
\begin{equation}\label{ellipticregularityL}
\|v\|_{H^2(\Omega)}\leq C_{\text{ell}}^\Omega
\Bigl(\|\mathbb{L}v\|_{L^2(\Omega)}
+\|v\|_{L^2(\Omega)}
\Bigr), \quad 
\text{ for all }v\in C^\infty(\Omega), \ Bv=0\text{ on }\partial\Omega. 
\end{equation}
\revision{which follows under sufficient regularity of the diffusion matrix $D\in C^1(\Omega)$ defining $\mathbb{L}$ as well as the domain $\partial\Omega\in C^2$, cf. e.g., \cite{GilbargTrudinger} 
}
\Margin{RevY (3)}

The Djirbashian-Caputo derivative in \eqref{ibvp} can be written as 
\[
(\partial_t^\alpha v)(t) = (\kernel^\alpha*v)(t)=\int_0^t \kernel^\alpha(t-s)v(s)\, ds
\text{ with }\kernel^\alpha(s)= \frac{1}{\Gamma(1-\alpha)} \, s^{-\alpha},
\]
and satisfies (among other properties) the coercivity estimate \cite[Lemma 1]{Alikhanov:10}
\begin{equation}\label{coercivityDC}
v(t)\,(\partial_t^\alpha v)(t)\geq\frac12 \partial_t^\alpha [v^2](t),
\qquad v\in W^{1,1}(0,T).
\end{equation}

\section{Main results}
\label{sec:mainresults}

Our goal is to prove power law / exponential decay rates on the difference between the solution $u$ to \eqref{ibvp} and a steady state $\uinf$
\begin{equation}\label{util}
\usim(t,x)=u(t,x)-\uinf(x),
\end{equation}
as well as on the time derivative
\begin{equation}\label{util_t}
\usim_t(t,x)=u_t(t,x) \ \text{ or }\ \partial_t^\alpha\usim(t,x)=\partial_t^\alpha u(t,x),
\end{equation}
provided $r$ tends to $\rinf$ fast enough. 
More precisely, we prove power law / exponential convergence 
\begin{equation}\label{decay_usim}
\|u(t)-\uinf\|_{H^s(\Omega)}^2\leq C \decayfunc(t) :=
C\begin{cases}t^{-\alpha} &\text{ in case }\alpha\in(0,1)\\
e^{-\omega t} &\text{ in case }\alpha=1\end{cases}
\quad\text{ for all }t\in(0,T),
\end{equation} 
for some $C$, $\omega>0$ independent of $T$ and $s\in\{0,1,2\}$, cf. Theorems~\ref{thm:decay_usim}, \ref{thm:decay_usim_nosmallness} (i).
In some applications, also decay of the time derivative to zero is relevant. 
It is, e.g., needed for proving contractivity of the fixed point operator $\mathcal{T}$ on which uniqueness as well as convergence of the iterative scheme for reconstructing $p(x)$, $q(x)$ from indirect observations in \cite{BB19} relies.
We here show
\begin{equation}\label{decay_u_t}
\|\partial_t^\alpha u(t)\|_{
H^s
(\Omega)}^2\leq C \decayfunc(t), \quad 
\|u_t(t)\|_{H^s(\Omega)}^2\leq C \decayfunc(t) 
\quad t>0
\end{equation} 
$s\in\{-1,0,1,2\}$, cf. Theorems~\ref{thm:decay_usim}, \ref{prop:decay_usim_t}, \ref{thm:decay_usim_nosmallness}, Remark~\ref{rem:decaydtalphau} under various different assumptions.

\medskip 

The essential tools in our proofs are 
\begin{itemize}
\item characterization of the difference $\usim$ and its time derivative as solutions to initial boundary value problems;
\item energy estimates, using the coercivity and regularity assumptions \eqref{ellipticityL}, \eqref{ellipticregularityL}, \eqref{coercivityDC}; 
\item use (and partly also development) of Gronwall type estimates; alternatively we also employ Tauberian theorems for quantifying the large time asymptotic behaviour.
\end{itemize}

We here focus on the practically relevant low regularity setting for the coefficients
\begin{equation}\label{LrrrLsss}
p\in L^\rrr(\Omega), \qquad q\in L^\sss(\Omega).
\end{equation}
While $\sss\geq2$ sufficies for all results here, we need some further restrictions, related to the growth of $f$, on $\rrr$ cf.  \eqref{growth}, \eqref{cond_r_kappa2_hi} below.

\medskip

Results related to \eqref{decay_usim}, 
\rerevisionown{besides \cite{LuchkoYamamoto:2025}, which we discuss in more detail in Remark~\ref{rem:LuchkoYamamoto2025} below,}
are, e.g., 
\cite[Sections 6.4.1 and 6.4.2]{Jin:2021}, \cite[Section 4.2.6]{BB2},
on the case $p\equiv1$, $q=0$, and 
\cite[Section 4.6]{KubicaRyszewskaYamamoto:2020}, 
where the latter is on non-autonomous but linear problems. 
However, the current setting is not covered by these results. 
In particular, the approach from \cite[Sections 6.4.1]{Jin:2021} (see also \cite[Section 4.2.6]{BB2}) of representing the solution via the variation of constants formula 
\[
u(t)=\mathbb{E}^\alpha(t)u_0+\int_0^t \mathbb{E}^\alpha(s)\left[-p\,f(u(s)+q\,u(s)+r(s)\right]\,ds  
\text{ in }L^2(\Omega)
\]
using the parametrized family of operators 
\[
(\mathbb{E}^\alpha(t)v)(x)=
\sum_{j\in\mathbb{N}} E_{\alpha,1}(-\lambda_j t^{\alpha})\langle v,\varphi_j\rangle_{L^2(\Omega)} \varphi_j(x)
\]
with the Mittag-Leffler function $E_{\alpha,1}$ and an eigensystem $(\lambda_j,\varphi_j)_{j\in\mathbb{N}}$ of the elliptic operator,
would require higher spatial regularity on the coefficients $p$ and $q$, that cannot be reproduced by the fixed point operator 
$\mathcal{T}$ 
on which the reconstruction scheme in \cite{BB19} relies.
It is thus important to be able to work in the low regularity coefficient setting \eqref{LrrrLsss}.

\medskip

Remarkable on the forward problem are the counteracting roles of the $p$ and $q$ terms, when having nonnegativity of these two coefficients as well as a monotonically increasing function $f$ in mind, as typical in applications.
Depending on the monotonicity behaviour of $f$ and on which of the terms dominates, this can result in exponential decay or blow-up of solutions 
\revision{(for the latter we point to \cite{FloridiaLiuYamamoto:2023} for the fractional and to \cite{QuittnerSouplet:2019} for the parabolic case)}
\Margin{RevY (2)}
or in convergence to a steady state $\uinf$ as time tends to infinity, that has to satisfy the elliptic PDE
\begin{equation}\label{bvp}
\begin{aligned}
-\mathbb{L}\uinf&=q\uinf-pf(\uinf)+\rinf \text{ in }\Omega \\ 
\Binf \uinf&=\ainf \text{ on }\partial\Omega. 
\end{aligned}
\end{equation}
Existence of a solution $\uinf\in H^2(\Omega)$ holds under mild conditions on the coefficients and the function $f$ from monotonicity arguments, see, e.g., \cite[Theorem 4.10]{Troeltzsch2010}.
\Margin{RevX 4.}
\revision{
\begin{proposition}
Let \eqref{ellipticityL}, \eqref{ellipticregularityL}, \eqref{LrrrLsss} hold with  $\rrr$, $\sss\geq 2$, $\rrr$, $\sss> \frac{d}{2}$, $f\in C^2(\Omega)$, $f(0)=0$, $f'\geq0$. Then for any $\rinf\in L^2(\Omega)$, 
$\ainf\in H^{1/2}(\partial\Omega)\cap L^{\tilde{s}}(\partial\Omega)$, $\tilde{s}>d-1$ in the Neumann case 
($\ainf\in H^{3/2}(\partial\Omega)\cap L^\infty(\partial\Omega)$ 
in the Dirichlet case), 
there exists a unique solution $\uinf\in H^2(\Omega)$ of \eqref{bvp}. 
\end{proposition}
\begin{proof}
Existence and uniqueness of a solution $\uinf\in H^1(\Omega)\cap L^\infty(\Omega)$ follows from \cite[Theorem 4.10]{Troeltzsch2010}. 
The higher regularity $\uinf\in H^2(\Omega)$ can be easily bootstrapped from \eqref{ellipticregularityL} and the fact that the right hand side of \eqref{bvp} is in $L^2(\Omega)$.
\end{proof}
}

Given a solution $\uinf$ of \eqref{bvp}, we impose  
a sign condition on the potential in the linearization of the steady state equation \eqref{bvp}
\begin{equation}\label{pinfty}
\pinf(x):=p(x)f'(\uinf(x))-q(x)\geq\underline{c}>0\quad x\in\Omega.
\end{equation}

Moreover, we require a growth condition of $f$
\begin{equation}\label{growth}
|f''(\xi)|\leq c_2+C_2|\xi|^{\kappa_2}, \quad
\xi\in\mathbb{R},
\end{equation}
a related requirement on the summability index $\rrr$ of the function $p$
\begin{equation}\label{cond_r_kappa2_lo}
d=1, \quad \rrr=\infty, \quad \kappa_2=1,
\end{equation}
or
\begin{equation}\label{cond_r_kappa2_hi}
\begin{aligned}
&d=1:\quad&& \rrr\geq2 \quad&& \kappa_2\in[0,\infty]\\
&d=2:\quad&& \rrr>2 \quad&& \kappa_2\in[0,\infty)\\
&d=3:\quad&& \rrr\geq6 \quad&& \kappa_2\in[0,2-\tfrac{6}{\rrr}],
\end{aligned}
\end{equation}
and a decay condition on $r-\rinf$ 
\begin{equation}\label{decay_r}
\|r(t)-\rinf\|_{X}^2\leq C_r 
\decayfunc(t)
\quad t>0
\end{equation} 
with $\decayfunc$ as in \eqref{decay_usim}, 
$\omega=
\omega_r
>0$ 
and $X\in\{H^{-1}(\Omega),\, L^2(\Omega)\}$. 

To simplify exposition, we restrict the setting to Dirichlet or Neumann conditions with time constant data, that is 
\begin{equation}\label{bndy_const}
a(t,x)=\ainf(x), \quad B u=\Binf u=\partial_\nu u\text{ or }B u=\Binf u=u, 
\end{equation}
but expect our results to remain valid, provided 
$a(t,x)\to \ainf(x)$, $B\to \Binf$, as $t\to\infty$ sufficiently fast. 
We also expect the analysis to carry over to more general boundary conditions, such as impedance or mixed conditions.
\\
Correspondingly we use $V=H^1(\Omega)$ or $V=H_0^1(\Omega)$ as function space (and by a slight abuse of notation, write $H^{-1}(\Omega)$ for its topological dual) for $u(t)$.

\medskip

\begin{theorem}\label{thm:decay_usim}
For $u_0\in H^1(\Omega)$, $f\in C^2(\mathbb{R})$, under  conditions \eqref{f0fprime0} \eqref{ellipticityL}, \eqref{ellipticregularityL},  
\eqref{pinfty}, \eqref{growth},  
$r\in L^\infty(0,\infty;X)$, \eqref{decay_r}, \eqref{bndy_const},
there exist $\rho_0>0$, $C>0$, 
independent of $T$ such that $\usim=u-\uinf$ (where $u$ solves \eqref{ibvp} and $\uinf$ solves \eqref{bvp}) satisfies the bound
\begin{equation}\label{bound_usim}
\begin{aligned}
&\frac{c}{2}\, \left(\kernel^{1-\alpha}*\|u-\uinf\|_{H^{s+1}(\Omega)}^2\right)(t)
+\|u(t)-\uinf\|_{H^s(\Omega)}^2\\
&\leq \|u(0)-\uinf\|_{H^s(\Omega)}^2 + \frac{C_r}{\Gamma(\alpha)\,\alpha(1-\alpha)} 
\end{aligned}
\end{equation} 
\revision{ for some constant $c>0$}
\Margin{RevX 5.}
as well as the decay estimate
\eqref{decay_usim}
provided 
\begin{equation}\label{smallness_init}
\|u(0)-\uinf\|_{H^s(\Omega)}<\rho_0, 
\end{equation}
with 
\begin{itemize}
\item[(lo)] $s=0$, $X=H^{-1}(\Omega)$ under condition \eqref{cond_r_kappa2_lo};
\item[(hi)] $s=1$, $X=L^2(\Omega)$ under condition  \eqref{cond_r_kappa2_hi}. 
\end{itemize}
In case 
$T=\infty$, we also have 
\begin{equation}\label{decay_usim_Tauber}
\|\partial_t^\alpha u(t)\|_{H^{s-1}(\Omega)}^2+\|u(t)-\uinf\|_{H^{s+1}(\Omega)}^2\lesssim C t^{-\alpha}
\quad \text{ as }t\to\infty.
\end{equation} 
Here $\eta(t) \lesssim C\,t^{-\gamma}$ as $t\to\infty$ means 
$\limsup_{t\to\infty} \eta(t)\, t^\gamma\leq C$.
\end{theorem}
\begin{proof}
See Section~\ref{sec:proof_decay_initialsmallness}.
\end{proof}

As a by-product, we also obtain existence of solutions to the forward problem.
\begin{corollary}\label{cor:wellposed_forward}
Assume that \eqref{bvp} has a solution $\uinf\in H^2(\Omega)$.
For any $T>0$, under the conditions of Theorem~\ref{thm:decay_usim} with $s=1$, 
a solution 
\begin{equation}\label{Ualpha}
u\in U^\alpha=
\begin{cases}
\{v\in C((0,T];H^2(\Omega)) \, : \partial_t^\alpha v\in C((0,T];L^2(\Omega))\}&\text{ if }\alpha<1\\
\{v\in C((0,T];H^2(\Omega)) \, : v_t\in C((0,T];H^1(\Omega))\}&\text{ if }\alpha=1
\end{cases}
\end{equation}
to \eqref{ibvp} exists and is unique.
\end{corollary}
\begin{proof}
To prove existence, one can follow the standard approch of Galerkin approximation, energy estimates  and weak limits (like in, e.g., \cite[Section 6.4.2]{Jin:2021}, but using the energy estimates in the proof of Theorem~\ref{thm:decay_usim}, that avoid differentiation of $p$ and $q$).
Uniqueness follows by the testing that we employ for obtaining lower order estimates, cf. \eqref{enest0_usim}. 
\end{proof}
\revision{
\begin{remark}\label{rem:LuchkoYamamoto2025}
For more comprehensive well-posedness result under more general assumptions, we refer to, e.g. \cite{GilbargTrudinger,LuchkoYamamoto:2025}.
While obtaining well-posedness is not the main task of our paper, we point to the fact that the regularity of solutions corresponds to the one stated in, e.g., \cite[Theorem 3.4]{LuchkoYamamoto:2025}, which might be viewed as a plausibility argument for our function space setting. 
\rerevisionown{More precisely, \cite[Example 3.3]{LuchkoYamamoto:2025} in case $p\in C^1(\Omega)$ provides a local in time well-posedness result for our setting.}
\\
Concerning decay, we point to \cite[Proposition 4.7.]{LuchkoYamamoto:2025}, where the better decay rate 
$\|u(t)-\uinf\|_{L^\infty(\Omega)}\lesssim C t^{-\alpha}$ (note the missing square on the norm, as compared to \eqref{decay_usim_Tauber}) has been proven in case $p(x)\equiv1$\footnote{taking into account the fact that $f'$ takes the opposite sign in the notation of \cite{LuchkoYamamoto:2025}} by means of comparison principles. 
\rerevision{The proof of this result relies on strict positivity of an eigenfunction $\phi_1$, which allows the authors to construct upper ond lower bound solutions as Mittag Leffler functions multiplied with $\phi_1$.
In our setting with the competing $q$ term, strict positivity of an eigenfunction is lost, though, and therefore we have to resort to different techniques for proving a decay result, namely, energy estimates that are obtained by an appropriate testing of the equation.
To remove the square on the norm also in our result, one might consider testing with $\usim \, |\usim|^{-1+\epsilon} \chi_{\{x\in\Omega\,:\,\usim(x)\not=0\}}$ rather than with $\usim$  
for $\epsilon\in(0,1]$, and study the limit $\epsilon\searrow0$. 
\footnote{More precisely, testing with $(\mu\usim-\mathbb{L}\usim) \, |\usim|^{-1+\epsilon} \chi_{\{x\in\Omega\,:\,\usim(x)\not=0\}}$ rather than with $\mu\usim-\mathbb{L}\usim$, see Section~\ref{sec:enest} below.} 
However, this would require a coercivity estimate of the form 
\begin{equation*}
\frac{v(t)}{|v(t)|^{1-\epsilon}}\,(\partial_t^\alpha v)(t)\geq c\partial_t^\alpha [|v|^{1+\epsilon}](t),
\qquad v\in W^{1,1}(0,T).
\end{equation*}
with some $c>0$ in place of \eqref{coercivityDC}, which clearly holds in case $\alpha=1$, but does not seem to be available (yet) for $\alpha<1$.
}
\MMargin{RefY (2)}
\end{remark}
}
\Margin{RevX 6.}
\Margin{RevY (4)}
Theorem~\ref{thm:decay_usim} already contains decay rates of $\partial_t^\alpha u$ in a low order spatial norm. 
Estimates in higher order norms can be established under further assumptions.
\begin{theorem}\label{prop:decay_usim_t} 
Let the conditions of Theorem~\ref{thm:decay_usim} and additionally 
\begin{equation}\label{utinit}
u_t(0)\in L^2(\Omega), \quad u_t(t_0)\in H^1(\Omega) \ \text{ for some }t_0\geq0
\end{equation}
be satisfied.
\begin{itemize}
\item[(a)] If $T=\infty$, $\alpha=1$ and 
$r_t\in L^2(0,\infty;L^2(\Omega))$, with $\|r_t(t)\|_{L^2(\Omega)}\to0$ as $t\to\infty$,
then  $\|u_t(t)\|_{H^1(\Omega)}\to0$ as $t\to\infty$.
\item[(b)] If $T=\infty$, $\alpha\in(0,1]$ and $\|r_t(t)\|_{L^2(\Omega)}^2=O(t^{-\beta})$ then there exists $C>0$ independent of $t_0$, such that
\[
\|u_t(t)\|_{H^2(\Omega)}^2\lesssim C (t-t_0)^{-\min\{\alpha,\beta\}}.
\]
\item[(c)] If 
$\|r_t(t)\|_{L^2(\Omega)}^2\leq C_r 
\decayfunc(t)$
for some $\omega:=
\omega_r
>0$, 
then there exist $C,\,\omega>0$ independent of $T$, $t_0$, such that 
\begin{equation}\label{expdecay_u_t}
\|u_t(t)\|_{H^2(\Omega)}^2\leq C 
\decayfunc(t-t_0)
\quad t\in(t_0,T)
\end{equation} 
\end{itemize}
\end{theorem}
\begin{proof}
See Section~\ref{sec:proof_decay_initialsmallness}.
\end{proof}
Note that according to the assumptions of Theorem~\ref{thm:decay_usim}, only smallness of the initial data $\|u(0)-\uinf\|_{H^1(\Omega)}<\rho_0$ for $\usim$ in the lower order norm is needed, but not smallness of the initial data for its time derivative $\partial_t^\alpha\usim=\partial_t^\alpha u$.
\\
We also point to the fact that in case $\alpha<1$, finiteness of $\|u_t(0)\|_{L^2(\Omega)}$ cannot be bootstrapped from the PDE. Thus we impose $u_t(0)\in L^2(\Omega)$ as an a priori regularity assumption here.
This a priori regularity condition on $u_t(0)$ 
\revision{cannot be concluded from regularity of the data alone. In particular,  not even in the homogeneous linear case $p=q=r=0$, it necessarily holds with general initial data $u_0\in H^2(\Omega)$, 
(unlike the case $\alpha=1$, where we immediately conclude 
$u_t(0)= \mathbb{L} u_0 - qu_0 + f(u_0)+r(0)\in L^2(\Omega)$).
This is due to the singularity of the Mittag-Leffler function $E_{\alpha,\alpha}$ at $t=0$.
For the same reason, no continuity at $t=0$ can be concluded in the $L^2(\Omega)$ norm and one has to resort to a weaker topology in space to conclude continuity of $u_t$ at $t=0$: From the PDE that $u_t$ satisfies (see  \eqref{timediffPDE} below with $\text{rhs}\in L^2(0,T;H^{-1}(\Omega))$), by an application of, e.g., \cite[Theorem 4.1]{KubicaRyszewskaYamamoto:2020} and with continuity of the embedding $H^\alpha(0,T;H^{-1}(\Omega))\to C([0,T];H^{-1}(\Omega))$ for $\alpha>\frac12$ we obtain $u_t(t)\to u_t(0)$ as $t\searrow0$ in $H^{-1}(\Omega)$.
\\ 
On the other hand, $u_t(0)\in L^2(\Omega)$ can still be satisfied under a certain interplay of source term and coefficients, as the simple example 
$r=\partial_t^\alpha \phi-\mathbb{L}\phi-q\phi+pf(\phi)$ 
for some $\phi\in C^1([0,T];H^2(\Omega))$ shows.
}
\Margin{RevX 7.}

\medskip

Decay estimates without imposing any closeness condition on the initial data to the steady state can be obtained by making an assumption of the type \eqref{pinfty} at the solution itself (rather than at the equilibrium state)
\begin{equation*}
p(x)f'(u(t,x))-q(x)\geq \underline{c}>0 \quad \text{ for a.e. }x\in\Omega, \ t\in(0,T),
\end{equation*}
or at all states between the transient and the equilibrium one 
\begin{equation}\label{boundpqf_int}
p(x)f'(\theta u(t,x)+(1-\theta)\uinf(x))-q(x)\geq \underline{c}>0 \quad \text{ for a.e. }x\in\Omega, \ t\in(0,T).
\end{equation}
This is clearly satisfied if $f$ is strictly monotonically increasing with derivative bounded away from zero and if $q$ is small enough as compared to $p$.
We only state the high regularity case $s=1$ corresponding to Theorem~\ref{thm:decay_usim} (hi) here; clearly a version of (lo) would be possible as well, but remain constrained by the restrictive condition \eqref{cond_r_kappa2_lo}.
\begin{theorem}\label{thm:decay_usim_nosmallness} 
Let $s=1$, $X=L^2(\Omega)$.
\begin{itemize}
\item[(i)] 
Under the assumptions of Theorem~\ref{thm:decay_usim} with \eqref{boundpqf_int} in place of \eqref{pinfty}, the assertions  of Theorem~\ref{thm:decay_usim} remain valid with $\rho_0=\infty$. 
\item[(ii)] If additionally \eqref{utinit} holds, then the assertions of Theorem~\ref{prop:decay_usim_t} hold.
\end{itemize}
\end{theorem}
\begin{proof}
See Section~\ref{sec:proof_decay_noinitialsmallness}.
\end{proof}

\begin{remark}\label{rem:decaydtalphau}
Decay of $\partial_t^\alpha$ in case $\alpha\in(\frac23,1)$ under the assumptions of Theorem~\ref{prop:decay_usim_t} 
\rerevisionown{(c)} 
(or its counterpart in Theorem~\ref{thm:decay_usim_nosmallness} (ii)) with $t_0=0$ can be derived by the simple estimate
\[
\|\partial_t^\alpha u(t)\|_{H^2(\Omega)}
\leq \int_0^t \kernel^\alpha(t-s)\|u_t(s)\|_{H^2(\Omega)}\, ds
\leq \frac{C}{\Gamma(1-\alpha)}\, t^{1-\alpha-\alpha/2}\int_0^1 (1-\sigma)^{-\alpha}\sigma^{\alpha/2}\, d\sigma.
\]
An improved decay estimate as in \cite[Section 4.6]{KubicaRyszewskaYamamoto:2020} would extend the range of possible $\alpha$ values to $\alpha\in(\frac12,1)$
\end{remark}

\begin{remark}
Under weaker assumptions on $f$ one can easily obtain just boundedness by a $T$ dependent constant. Imposing in place of \eqref{pinfty}
\[
p(x)\geq0\quad x\in\Omega, \quad q\in L^\infty(\Omega), \quad f(\xi)\xi\geq0\quad \xi\in\mathbb{R},
\]
we directly test \eqref{ibvp} with $u$ to obtain, using \eqref{ellipticityL}, \eqref{coercivityDC},
\[
\begin{aligned}
&\frac12\partial_t^\alpha\|u(t)\|_{L^2(\Omega)}^2
+\underline{c}_D\, \|\nabla u(t)\|_{L^2(\Omega)}^2
\\
&\leq \|q\|_{L^\infty(\Omega)} \|u(t)\|_{L^2(\Omega)}^2
+\int_\Omega r(t)\,u(t)\, dx,
\end{aligned}
\]
hence 
\[
\begin{aligned}
&\partial_t^\alpha\|u(t)\|_{L^2(\Omega)}^2
+\underline{c}_D\, \|u(t)\|_{H^1(\Omega)}^2
\\
&\leq 2(\|q\|_{L^\infty(\Omega)} +\underline{c}_D) \|u(t)\|_{L^2(\Omega)}^2
+\frac{1}{\underline{c}_D} \, \|r(t)\|_{H^{-1}(\Omega)}^2,
\end{aligned}
\]
and apply Gronwall's inequality (see, e.g., \cite[Theorem 4.2]{Jin:2021}, \cite[Lemma A.1.]{BBB}) to obtain 
\[
\|u\|_{L^\infty(0,T;L^2(\Omega))}^2+\|u\|_{L^2(0,T;H^1(\Omega))}^2
\leq C_T(\|u^0\|_{L^2(\Omega)}^2+\|r\|_{L^2(0,T;H^{-1}(\Omega))}^2)
\]
\end{remark}

\section{Proof of Theorems~\ref{thm:decay_usim} and \ref{prop:decay_usim_t}}\label{sec:proof_decay_initialsmallness}
To exploit conditon \eqref{pinfty}, we write the difference \eqref{util} of solutions \eqref{ibvp}  and \eqref{bvp} with \eqref{bndy_const} as a solution to the initial boundary value problem
\begin{equation}\label{ibvp_usim}
\begin{aligned}
\partial_t^\alpha u^\sim+(-\mathbb{L}+\pinf)\usim&=
-p\,\frac12\int_0^1f''(\uinf+\theta\usim)\,d\theta\,(\usim)^2+r-\rinf \text{ in }(0,T)\times\Omega, \\ 
B \usim&=0 \text{ on }(0,T)\times\partial\Omega, \\
\usim&=u^0-\uinf\text{ in }\{0\}\times\Omega.
\end{aligned}
\end{equation}
where we have used 
\[
f(\uinf+\usim)-f(\uinf)-f'(\uinf)\usim=
\frac12\int_0^1f''(\uinf+\theta\usim)\,d\theta\,(\usim)^2.
\]
\subsection{Energy estimates}\label{sec:enest}
To obtain energy estimates on $u^\sim$ we multiply the PDE in \eqref{ibvp_usim} with 
$\usim$ and integrate over $\Omega$ to obtain, using \eqref{coercivityDC},
that
\[
\begin{aligned}
&\frac12\partial_t^\alpha\|\usim(t)\|_{L^2(\Omega)}^2
+\underline{c}_D\, \|\nabla\usim(t)\|_{L^2(\Omega)}^2
+\underline{c}\, \|\usim(t)\|_{L^2(\Omega)}^2
\\
&\leq\int_\Omega\Bigl(-p\,\frac12\int_0^1f''(\uinf+\theta\usim(t))\,d\theta\,(\usim(t))^2+r(t)-\rinf
\Bigr)\,\usim(t)\, dx,
\end{aligned}
\]
thus, using Young's inequality and multiplying by two
\begin{equation}\label{enest0_usim}
\begin{aligned}
&\partial_t^\alpha\|\usim(t)\|_{L^2(\Omega)}^2
+\min\{\underline{c},\,\underline{c}_D\}\, \|\usim(t)\|_{H^1(\Omega)}^2
\\
&\leq 
\frac{1}{\min\{\underline{c},\,\underline{c}_D\}}\,
\|p\,\frac12\int_0^1f''(\uinf+\theta\usim(t))\,d\theta\,(\usim(t))^2-(r(t)-\rinf)\|_{H^{-1}(\Omega)}^2.
\end{aligned}
\end{equation}
A higher regularity estimate can be derived by testing with $\mu\usim-\mathbb{L}\usim$ for some $\mu>0$,  
which via Young's inequality yields 
\begin{equation}\label{enid_usim}
\begin{aligned}
&\frac12\mu\partial_t^\alpha\|\usim(t)\|_{L^2(\Omega)}^2
+\frac12\underline{c}_D\partial_t^\alpha\|\nabla\usim(t)\|_{L^2(\Omega)}^2
\\&\qquad 
+(\mu\underline{c}-\frac12\|\pinf\|_{L^\infty(\Omega)}^2)\, \|\usim(t)\|_{L^2(\Omega)}^2
+\frac12\|\mathbb{L}\usim(t)\|_{L^2(\Omega)}^2
\\
&\leq\int_\Omega\Bigl(-p\,\frac12\int_0^1f''(\uinf+\theta\usim(t))\,d\theta\,(\usim(t))^2+r(t)-\rinf
\Bigr)\,(\mu\usim-\mathbb{L}\usim)\, dx\\
&\leq\left(\frac{\mu}{2\underline{c}}+1\right)
\|p\,\frac12\int_0^1f''(\uinf+\theta\usim(t))\,d\theta\,(\usim(t))^2-(r(t)-\rinf)\|_{L^2(\Omega)}^2
\\&\qquad 
+\frac{\mu\underline{c}}{2} \|\usim(t)\|_{L^2(\Omega)}^2
+\frac14\|\mathbb{L}\usim(t)\|_{L^2(\Omega)}^2.
\end{aligned}
\end{equation}
Note that we do \textit{not} integrate the term contaning $\pinf$
 by parts, since this would lead to a gradient of $\pinf$, but estimate it by 
\[\begin{aligned}
\left|\int_\Omega\pinf\usim(t)\, \mathbb{L}\usim(t)\, dx\right|
&\leq \|\pinf\|_{L^\infty(\Omega)}\, \|\usim(t)\|_{L^2(\Omega)}\,\|\mathbb{L}\usim(t)\|_{L^2(\Omega)}\\
&\leq  \frac12\|\pinf\|_{L^\infty(\Omega)}^2\, \|\usim(t)\|_{L^2(\Omega)}^2+\frac12\|\mathbb{L}\usim(t)\|_{L^2(\Omega)}^2.
\end{aligned}\]
Thus, using \eqref{ellipticregularityL}, and multiplying by two we get 
\[
\begin{aligned}
&\min\{\mu,\, \underline{c}_D\} \partial_t^\alpha\|\usim(t)\|_{H^1(\Omega)}^2
+\frac{\min\{\mu\,\underline{c}-\|\pinf\|_{L^\infty(\Omega)}^2,\,\frac12\}}{(C_{\text{ell}}^\Omega)^2}\, 
\|\usim(t)\|_{H^2(\Omega)}^2
\\
&\leq 
\left(\frac{\mu}{\underline{c}}+2\right)\,
\|p\,\frac12\int_0^1f''(\uinf+\theta\usim(t))\,d\theta\,(\usim(t))^2-(r(t)-\rinf)\|_{L^2(\Omega)}^2
\end{aligned}
\]
where we choose, e.g. $\mu=\frac{1}{2\underline{c}}(1+2\|\pinf\|_{L^\infty(\Omega)}^2)$ to obtain
\begin{equation}\label{enest_usim}
\begin{aligned}
&\min\left\{\frac{1+2\|\pinf\|_{L^\infty(\Omega)}^2}{2\underline{c}},\, \underline{c}_D\right\} \partial_t^\alpha\|\usim(t)\|_{H^1(\Omega)}^2
+\frac{1}{2(C_{\text{ell}}^\Omega)^2}\, \|\usim(t)\|_{H^2(\Omega)}^2
\\
&\leq 
\left(\frac{1+2\|\pinf\|_{L^\infty(\Omega)}^2}{2\underline{c}^2}+2\right)\,
\|p\,\frac12\int_0^1f''(\uinf+\theta\usim(t))\,d\theta\,(\usim(t))^2-(r(t)-\rinf)\|_{L^2(\Omega)}^2.
\end{aligned}
\end{equation}
  
The right hand sides in \eqref{enest0_usim}, \eqref{enest_usim} can be further estimated by imposing the growth condition \eqref{growth} on $f$.
Relying on the continuous embeddings and duality 
\begin{equation}\label{embeddings}
\begin{aligned}
&H^1(\Omega)\to L^\ell(\Omega) \text{  and 
}L^m(\Omega)\to H^{-1}(\Omega)  \text{ for } 
\begin{cases} 
\ell\leq\infty, \ m\geq1 &\text{ if }d=1\\
\ell<\infty, \ m>1 &\text{ if }d=2\\
\ell\leq6, \ m\geq\frac65 &\text{ if }d=3
\end{cases}
\\
&H^2(\Omega)\to L^\infty(\Omega), 
\end{aligned}
\end{equation}
we estimate, in the low regularity scenario \eqref{enest0_usim} 
\[\begin{aligned}
&\|p\int_0^1f''(\uinf+\theta\usim(t))\,d\theta\,(\usim(t))^2\|_{L^m(\Omega)}\\
&\leq 
\|p\|_{L^{\rrr}(\Omega)}
\|\int_0^1f''(\uinf+\theta\usim(t))\,d\theta\,\usim(t)\|_{L^\nu(\Omega)}
\|\usim(t)\|_{L^\ell(\Omega)}
\\
&\leq 
\|p\|_{L^{\rrr}(\Omega)}\Bigl(
c_2 \|\usim(t)\|_{L^\nu(\Omega)}
+ C_2\,C_{\kappa_2}\,\bigl(\|\,|\uinf|^{\kappa_2+1}\|_{L^\nu(\Omega)}+\|\,|\usim(t)|^{\kappa_2+1}\|_{L^\nu(\Omega)}\bigr)\Bigr)
\|\usim(t)\|_{L^\ell(\Omega)}
\\
&= 
\|p\|_{L^{\rrr}(\Omega)}\Bigl(
c_2 \|\usim(t)\|_{L^\nu(\Omega)}
+ C_2\,C_{\kappa_2}\,\bigl(\|\uinf\|_{L^{(\kappa_2+1)\nu}(\Omega)}^{\kappa_2+1}+\|\usim(t)\|_{L^{(\kappa_2+1)\nu}(\Omega)}^{\kappa_2+1}\bigr)\Bigr)
\|\usim(t)\|_{L^\ell(\Omega)}\\
&\leq \|p\|_{L^{\rrr}(\Omega)}\bigl(\Phi_\infty^0(\|\uinf\|_{L^2(\Omega)})+\Phi_\sim^0(\|\usim(t)\|_{L^2(\Omega)})\bigr)
\|\usim(t)\|_{L^2(\Omega)}\|\usim(t)\|_{H^1(\Omega)}
\end{aligned}\]
with
\begin{equation}\label{Phi0Phisim_0}
\begin{aligned}
&\Phi_\sim^0(\xi)=C_2\,C_{\kappa_2}\,(C_{L^2\to L^{(\kappa_2+1)\nu}}^\Omega)^{\kappa_2+1}\,C_{H^1\to L^\ell}^\Omega\,\xi^{\kappa_2}\\
&\Phi_\infty^0(\xi)=\Phi_\sim^0(\xi)+c_2\,C_{L^2\to L^\nu}^\Omega\,C_{H^1\to L^\ell}^\Omega,
\end{aligned}
\end{equation}
where we have used the generalized H\"older inequality with
\begin{equation}\label{rlnum}
\frac{1}{\rrr}+\frac{1}{\ell}+\frac{1}{\nu}=\frac{1}{m}
\end{equation}
as well as
\[
(a+b)^{\kappa_2}\leq C_{\kappa_2}(a^{\kappa_2}+b^{\kappa_2}), \quad a,\, b\,\geq0 \quad\text{ for }C_{\kappa_2}=\max\{1,2^{\kappa_2-1}\}.
\]
This leads to the requirements $\rrr>m$, $\nu\leq2$, $(\kappa_2+1)\nu\leq2$, thus, via \eqref{rlnum}
\[
\frac{1}{\rrr}+\frac{1}{\ell}+\frac{\kappa_2+1}{2}\leq\frac{1}{m},
\]
which can only be satisfied in the case 
\begin{equation*}
d=m=1, \quad \ell=\rrr=\infty, \quad \kappa_2=1,
\end{equation*}
that is, \eqref{cond_r_kappa2_lo},
that includes the Allen-Cahn nonlinearity $f(\xi)=\xi^3$, but is restricted to one space dimension.

In the high regularity scenario \eqref{enest_usim}, we estimate
\begin{equation}\label{estgrowth_m2}
\begin{aligned}
&\|p\int_0^1f''(\uinf+\theta\usim)\,d\theta\,(\usim)^2\|_{L^2(\Omega)}\\
&\leq 
\|p\|_{L^{\rrr}(\Omega)}
\|\int_0^1f''(\uinf+\theta\usim(t))\,d\theta\,\usim(t)\|_{L^{2\rrr/(\rrr-2)}(\Omega)}
\|\usim(t)\|_{L^\infty(\Omega)}
\\
&\leq 
\|p\|_{L^{\rrr}(\Omega)}\Bigl(
c_2 \|\usim(t)\|_{L^{2\rrr/(\rrr-2)}(\Omega)}
\\
&\phantom{\leq\|p\|_{L^{\rrr}(\Omega)}\Bigl(}
+ C_2\,C_{\kappa_2}\,\bigl(\|\uinf\|_{L^{(\kappa_2+1){2\rrr/(\rrr-2)}}(\Omega)}^{\kappa_2+1}+\|\usim(t)\|_{L^{(\kappa_2+1){2\rrr/(\rrr-2)}}(\Omega)}^{\kappa_2+1}\bigr)\Bigr)
\|\usim(t)\|_{L^\infty(\Omega)}\\
&\leq \|p\|_{L^{\rrr}(\Omega)}\bigl(\Phi_\infty^1(\|\uinf\|_{H^1(\Omega)})+\Phi_\sim^1(\|\usim(t)\|_{H^1(\Omega)})\bigr)
\|\usim(t)\|_{H^1(\Omega)}\|\usim(t)\|_{H^2(\Omega)}
\end{aligned}
\end{equation}
with
\begin{equation}\label{Phi0Phisim}
\begin{aligned}
&\Phi_\sim^1(\xi)=C_2\,C_{\kappa_2}\,(C_{H^1\to L^{(\kappa_2+1){2\rrr/(\rrr-2)}}}^\Omega)^{\kappa_2+1}\,C_{H^2\to L^\infty}^\Omega\,\xi^{\kappa_2}\\ &\Phi_\infty^1(\xi)=\Phi_\sim^0(\xi)+c_2\,C_{H^1\to L^{2\rrr/(\rrr-2)}}^\Omega\,C_{H^2\to L^\infty}^\Omega.
\end{aligned}
\end{equation}
This leads to the requirements 
\[
\rrr\geq2 \text{ and }\frac{2\rrr}{\rrr-2}\leq\ell
\text{ and }\frac{(\kappa_2+1)2\rrr}{\rrr-2}\leq\ell,
\]
that is, \eqref{cond_r_kappa2_hi},
which even in three space dimensions with $\rrr=6$ includes the Allen-Cahn nonlinearity $\kappa_2=1$, $f(\xi)=\xi^3$, and with $\rrr=\infty$ even allows for $\kappa_2=2$, \revision{$f(\xi)=|\xi|^3\,\xi$}; 
\Margin{RevX 8.}
in one space dimension we expect even exponential nonlinearities to work out (although we do not include them formally here).

By means of \eqref{enest0_usim}, \eqref{enest_usim}, we have thus shown the following.
\begin{proposition}\label{prop:enests}
For $p\in L^{\rrr}(\Omega)$, $f\in C^2(\mathbb{R})$, under conditions \eqref{ellipticityL}, \eqref{ellipticregularityL}, 
\eqref{pinfty}, \eqref{growth}, 
the solution $\usim=u-\uinf$ to \eqref{ibvp_usim} (where $u$ solves \eqref{ibvp} and $\uinf$ solves \eqref{bvp}) satisfies the estimates
\begin{itemize}
\item[(lo)]  with \eqref{cond_r_kappa2_lo},
\begin{equation}\label{enest_usim_lo}
\begin{aligned}
&\partial_t^\alpha\|\usim(t)\|_{L^2(\Omega)}^2
+c\, \|\usim(t)\|_{H^1(\Omega)}^2 \bigl(1-D^0(t)\bigr)
\leq \tilde{C}\|r(t)-\rinf\|_{H^{-1}(\Omega)}^2\\
&\text{with }D^0(t)=C\|p\|_{L^{\rrr}(\Omega)}^2\bigl(\Phi_\infty^0(\|\uinf\|_{L^2(\Omega)})+\Phi_\sim^0(\|\usim(t)\|_{L^2(\Omega)})\bigr)^2\|\usim(t)\|_{L^2(\Omega)}^2
\end{aligned}
\end{equation} 
\item[(hi)]  with \eqref{cond_r_kappa2_hi} 
\begin{equation}\label{enest_usim_hi}
\begin{aligned}
&\partial_t^\alpha\|\usim(t)\|_{H^1(\Omega)}^2
+c\, \|\usim(t)\|_{H^2(\Omega)}^2 \bigl(1-D^1(t)\bigr)
\leq \tilde{C}\|r(t)-\rinf\|_{L^2(\Omega)}^2\\
&\text{with }D^1(t)=C\|p\|_{L^{\rrr}(\Omega)}^2\bigl(\Phi_\infty^1(\|\uinf\|_{H^1(\Omega)})+\Phi_\sim^1(\|\usim(t)\|_{H^1(\Omega)})\bigr)^2
\|\usim(t)\|_{H^1(\Omega)}^2,
\end{aligned}
\end{equation} 
\end{itemize}
where the constants $c$, $C$, $\tilde{C}$ are independent of $t$ 
and $\Phi_\infty^s$, $\Phi_\sim^s$, $s\in\{0,1\}$ are defined according to \eqref{Phi0Phisim_0}, \eqref{Phi0Phisim}. 
\end{proposition}
\noindent
Note that other than $p\in L^{\rrr}(\Omega)$ and \eqref{pinfty} we do not need to impose any regularity assumptions on $p$ and $q$, once we have existence of a solution to \eqref{bvp}.

\subsection{Gronwall type estimates}
The estimates from Proposition~\ref{prop:enests} allow us to prove decay rates $\|\usim(t)\|^2=O(\decayfunc(t))$ with $\decayfunc$ as in \eqref{decay_usim}
as well known in case $\alpha=1$, whereas in case $\alpha\in(0,1)$ this can be deduced from results in, e.g., \cite[Chapter 5]{KubicaRyszewskaYamamoto:2020}.
(Note that there even better decay rates are shown; however, without inhomogeneity.) 
For the convenience of the reader, we summarize this in a lemma and provide its proof in the appendix.
\begin{lemma}\label{lem:eta}
Let $\alpha\in(0,1)$,
$\eta\in W^{1,1}(0,T)$ and assume that for any $\underline{t}\in(0,T)$, there exists $\beta\in(\alpha,1)$ such that $\eta\in W^{1,1/(1-\beta)} (\underline{t},T)$,
and that 
\begin{equation}\label{enest_eta}
\partial_t^\alpha\eta(t)+c_0\eta(t) \leq \phi(t) \text{ for almost every }t\in(0,T)
\end{equation}
holds for some $c_0>0$.
Then with the Mittag-Leffler functions $E_{\alpha,1}$, $E_{\alpha,\alpha}$, the estimate  
\begin{equation}\label{eta_leq_nu}
\begin{aligned}
&\eta(t) \leq \nu(t):=E_{\alpha,1}(-c_0t^\alpha)\, \eta(0) + \int_0^t (t-s)^{\alpha-1}E_{\alpha,\alpha}(-c_0(t-s)^\alpha)\phi(s)\, ds\\
&\text{ for all }t\in(0,T]
\end{aligned}
\end{equation}
is satisfied.\\
Moreover, if $\eta$ is nonnegative, $\eta(t)\geq0$, for almost every $t\in(0,T)$, then
\begin{equation}\label{bound_eta_phi}
\eta(t) \leq \eta(0) + \frac{1}{\Gamma(\alpha)} \, \int_0^t (t-s)^{\alpha-1}\,\phi(s)\, ds
\text{ for all }t\in(0,T].
\end{equation}
In particular, the premiss \eqref{enest_eta} with $\phi(t)=\phi_0 t^{-\alpha}$ implies
\begin{equation}\label{powerdecay_eta}
\eta(t)\leq \Bigl(\Gamma(1+\alpha)c_0^{-\alpha}\,\eta(0)+\frac{\phi_0\,C(\alpha)}{c_0}\Bigr) \, t^{-\alpha}
\text{ for all }t\in(0,T]
\end{equation}
with a constant $C(\alpha)>0$ depending only on $\alpha$
and, in case of nonnegative $\eta$,
\begin{equation}\label{bound_eta}
\eta(t) \leq \eta(0) + \frac{\phi_0}{\Gamma(\alpha)\,\alpha(1-\alpha)}
\text{ for all }t\in(0,T];
\end{equation}
the premiss \eqref{enest_eta} with $\phi(t)\equiv \phi_0$ implies
\begin{equation}\label{bound_eta_0}
\eta(t) \leq \eta(0)+\frac{\phi_0}{c_0}
\text{ for all }t\in(0,T].
\end{equation}

In case $\alpha=1$, for any $\eta\in W^{1,1}(0,T)$, estimate \eqref{enest_eta} implies \eqref{eta_leq_nu} with $E_{1,1}=\exp$. In particular, for $\phi(t)=\phi_0 e^{-c_1 t}$ with $c_1>c_0$ we have 
\begin{equation}\label{expdecay_eta}
\eta(t)\leq \Bigl(\eta(0)+\frac{\phi_0}{c_1-c_0}\Bigr) \, e^{-c_0 t}
\text{ for all }t\in(0,T].
\end{equation}
\end{lemma} 
\begin{proof}
See the appendix.
\end{proof}

The following Lemma will also allow us to conclude decay of the higher order norm $\|\usim(t)\|_{H^2(\Omega)}$ from \eqref{enest_usim_hi}.
As compared to the exact estimates from Lemma~\ref{lem:eta}, we only obtain this in an asymptotic sense for $\alpha<1$.
Recall that $\eta(t) \lesssim C\,t^{-\gamma}$ as $t\to\infty$ means 
$\limsup_{t\to\infty} \eta(t)\, t^\gamma\leq C$.

\begin{lemma}\label{lem:eta01}
Let $\alpha\in(0,1]$, let $\eta_0:[0,T)\to[0,\infty)$, $\eta_1:\,(0,T)\to[0,\infty)$ be nonnegative measureable functions, and assume that
\begin{equation}\label{enest_eta01}
\partial_t^\alpha\eta_0(t)+c_0\eta_0(t)+c_1\eta_1(t) \leq C t^{-\beta} \text{ for almost every }t\in(0,T)
\end{equation}
holds for some $c_0,\,c_1\,\beta\in(0,1]$.
Then if $T=\infty$ 
\begin{equation}\label{decay_eta01}
\eta_j(t)\lesssim C_j\,t^{-\min\{\alpha,\beta\}}  \quad \text{ as }t\to\infty \qquad j\in\{0,1\}
\end{equation}
holds with $C_j=\frac{C\Gamma(1-\beta)+\eta_0(0)}{c_j}$. 

If $\alpha=1$ then 
\begin{equation}\label{enest_eta01_exp}
\partial_t\eta_0(t)+c_0\eta_0(t)+c_1\eta_1(t) \leq C e^{-c_0t} \text{ for almost every }t\in(0,T)
\end{equation}
implies
\begin{equation}\label{decay_eta01_exp}
\eta_0(t)\leq (\eta_0(0)+Ct) e^{-c_0t} \quad 
\eta_1(t)\leq \tilde{C} e^{-\tilde{c}_0t} \quad 
\text{ for almost every }t\in(0,T)
\end{equation}
for some $\tilde{C}>0$, $0<\tilde{c}_0<c_0$.
\end{lemma}
\begin{proof}
Applying the Laplace transform $(\mathcal{L}v)(s)=\int_0^\infty e^{-st} v(t)\,dt$ to both sides of \eqref{enest_eta01} yields
\[
(s^\alpha +c_0)(\mathcal{L}\eta_0)(s)+c_1(\mathcal{L}\eta_1)(s)
\leq C\Gamma(1-\beta)\,s^{\beta-1}+ \eta_0(0)\,s^{\alpha-1}\quad s>0
\]
and  $\mathcal{L}\eta_0$, $\mathcal{L}\eta_1$ are well-defined on the right half plane.
Thus, in particular we have the following asyptotics near zero 
\[
(\mathcal{L}\eta_j)(s)\lesssim  \frac{C\Gamma(1-\beta)+\eta_0(0)}{c_j}\, s^{\min\{\alpha,\beta\}-1}, \quad \text{ as }s\to0\qquad j\in\{0,1\}.
\]
A Tauberian Theorem (e.g., \cite[Theorem 3.14]{BBB} with $dw=\eta_j(t)\,dt$), which is also valid for $\alpha=1$, thus yields
\[
\eta_j(t)\lesssim  \frac{C\Gamma(1-\beta)+\eta_0(0)}{c_j}\, t^{-\min\{\alpha,\beta\}} 
\quad \text{ as }t\to\infty\qquad j\in\{0,1\}.
\]
In case $\alpha=1$ we obtain from \eqref{enest_eta01_exp}, after multiplying with $e^{c_0t}$ and integration
\[
0\leq\eta_0(t)e^{c_0t}\leq \eta_0(t_0)+ C(t-t_0) -c_1\int_{t_0}^t \eta_1(s)e^{c_0s}\,ds\qquad 0<t_0 < t,
\]
hence, with $t_0=0$,
$\eta_0(t)\leq (\eta_0(0)+Ct)e^{-c_0t}$ and,
after rearranging,
\[
c_1\int_{t_0}^t e^{c_0s}\eta_1(s)\,ds\leq \eta_0(t_0)+C(t-t_0) \qquad 0<t_0 < t,
\]
which implies existence of $\tilde{C}>0$, $0<\tilde{c}_0<c_0$ 
such that
\[
\eta_1(t)\leq \tilde{C} e^{-\tilde{c}_0t}.
\]
\end{proof}

Under a power law / exponential decay condition \eqref{decay_r} on $r-\rinf$ 
with $\decayfunc$ as in \eqref{decay_usim}, $\omega=c_1>0$ with $X\in\{H^{-1}(\Omega),\, L^2(\Omega)\}$ 
and a smallness condition on the initial difference, this allows us to prove power law / exponential  decay of $\usim$. 
\subsection{Decay of $\usim$; proof of Theorem~\ref{thm:decay_usim}}\label{subsec:decayusim}
Again, we focus on the case $\alpha<1$.\\
To tackle the negative power law terms on the left hand side of the estimates \eqref{enest_usim_lo}, \eqref{enest_usim_hi}, we apply barrier's method to show the following
(which we will use with $\eta(t)=\|\usim(t)\|_{H^s(\Omega)}^2$, $\tilde{\eta}(t)=\|\usim(t)\|_{H^{s+1}(\Omega)}^2$ and $\Phi$ as in \eqref{Phi} below):
\begin{quote}
For $\alpha\in(0,1)$, $\eta$ nonnegative as in Lemma~\ref{lem:eta} the premiss
\begin{equation}\label{enest_eta_Phi}
\begin{aligned}
&\partial_t^\alpha\eta(t)+c_0\eta(t)+c_1\tilde{\eta}(t)(1-\Phi(\eta(t)))\leq c_2 t^{-\alpha} 
\text{ for almost every }t\in(0,T)\\
&\text{ and } \Phi\left(\eta(0) + \frac{c_2}{\Gamma(\alpha)\,\alpha(1-\alpha)}\right)<1
\end{aligned}
\end{equation}
with some nonnegative functions $\tilde{\eta}\geq0$, $\Phi\geq0$, where the latter is nondecreasing, and constants $c_0$, $c_1$, $c_2$ $>0$, implies 
\begin{equation}\label{Phismaller1}
\Phi(\eta(t))\leq1 \text{ for all }t\in(0,T]
\end{equation}
and thus, via Lemma~\ref{lem:eta}, \eqref{powerdecay_eta}.
\end{quote}
Assume contrarily that there exists a time $t_*\in(0,T]$ such that $\Phi(\eta(t_*))>1$; without loss of generality, $t_*$ is the minimal such time. As a consequence, 
\[
\partial_t^\alpha\eta(t)+c_0\eta(t)\leq c_2 t^{-\alpha} \text{ for almost every }t\in(0,t_*)
\]
and thus by \eqref{bound_eta} in Lemma~\ref{lem:eta} and monotonicity of $\Phi$
\[
\Phi(\eta(t_*)) \leq \Phi\left(\eta(0) + \frac{c_2}{\Gamma(\alpha)\,\alpha(1-\alpha)}\right) \,<1,
\]
a contradiction. 

In order to establish \eqref{enest_eta_Phi} with $\eta(t)=\|\usim(t)\|_{H^s(\Omega)}^2$, $\tilde{\eta}(t)=\|\usim(t)\|_{H^{s+1}(\Omega)}^2$ from \eqref{enest_usim_lo} or \eqref{enest_usim_hi}, we split 
\[
\begin{aligned}
&c\, \|\usim(t)\|_{H^{s+1}(\Omega)}^2 \bigl(1-D^s(t)\bigr)\\
&\geq \frac{c}{2}\, \|\usim(t)\|_{H^s(\Omega)}^2
+ \frac{c}{2}\, \|\usim(t)\|_{H^{s+1}(\Omega)}^2 \bigl(1-2D^s(t)\bigr)
\end{aligned}
\]
for $s\in\{0,1\}$, and set $c_0=c_1=\frac{c}{2}$, 
\begin{equation}\label{Phi}
\Phi(\eta)=
C\|p\|_{L^{\rrr}(\Omega)}^2\bigl(\Phi_\infty^s(\|\uinf\|_{H^s(\Omega)})+\Phi_\sim^s(\sqrt{\eta})\bigr)^2\,\eta.
\end{equation}

The implication \eqref{enest_eta_Phi} $\Rightarrow$ \eqref{Phismaller1} remains valid with the upper bound $1$ replaced by $\frac12$, thus implying 
\[
\partial_t^\alpha(\eta(t)-\eta(0))+c_0\eta(t)+\frac{c_1}{2}\tilde{\eta}(t)\leq c_2 t^{-\alpha} 
\text{ for almost every }t\in(0,T).
\]
Convolving both sides with $\kernel^{1-\alpha}$ we obtain (like in \eqref{etak1malpha}, \eqref{estconst})
\[
\eta(t)-\eta(0)+c_0(\kernel^{1-\alpha}*\eta)(t)+\frac{c_1}{2}(\kernel^{1-\alpha}*\tilde{\eta})(t)
\leq \frac{c_2}{\Gamma(\alpha)}\left(\frac{1}{1-\alpha}+\frac{1}{\alpha}\right)\,, 
\]
hence \eqref{bound_usim}.

In case $\alpha\in(0,1)$, $T=\infty$, we use Lemma~\ref{lem:eta01} to first of all obtain $\|\usim(t)\|_{H^{s+1}(\Omega)}^2\lesssim C t^{-\alpha}$ as $t\to\infty$ and then, by inserting into the PDE \eqref{ibvp_usim}, also 
\[
\begin{aligned}
&\|\partial_t^\alpha u^\sim(t)\|_{H^{s-1}(\Omega)}\\
&=\|(-\mathbb{L}+\pinf)\usim(t)-p\,\frac12\int_0^1f''(\uinf+\theta\usim(t))\,d\theta\,(\usim)^2+r(t)-\rinf\|_{H^{s-1}(\Omega)}\\
&\lesssim\|\usim(t)\|_{H^{s+1}(\Omega)}+\|r(t)-\rinf\|_{H^{s-1}(\Omega)}
\end{aligned}
\]
thus \eqref{decay_usim_Tauber}.
 

\subsection{Decay of $u_t$; proof of Theorem~\ref{prop:decay_usim_t}}\label{subsec:decayut}
Estimate \eqref{decay_usim_Tauber} already provides a decay rate on the time derivative in a relatively weak spatial norm.
To obtain a decay estimate in a stronger norm, 
we also consider the time differentiated version of the PDE in \eqref{ibvp_usim} for 
$\partial_t^{\altil} u^\sim=\partial_t^{\altil} u$ for $\altil\in\{\alpha,1\}$ 
\begin{equation}\label{timediffPDE}
\begin{aligned}
&(\partial_t^\alpha -\mathbb{L}+\pinf)\partial_t^{\altil}\usim
=\text{rhs}\\
&\text{ with }\ \text{rhs}=\begin{cases}
-p\,\kernel^{\altil}*\Bigl(\int_0^1f''(\uinf+\theta\usim)\,d\theta\usim\,\usim_t\Bigr)+\partial_t^{\altil} r&\text{ if }\altil=\alpha\in(0,1]\\
-p\,\Bigl(\int_0^1f''(\uinf+\theta\usim)\,d\theta\usim\,\usim_t\Bigr)+r_t-\kernel^\alpha(t)u_t(0)&\text{ if }\altil=1,  \ \alpha\in(0,1)\\
\end{cases}
\end{aligned}
\end{equation}
where we formally set $\kernel^1*:=\text{id}$ to include $\alpha=1$ in the first case and used 
\[
\begin{aligned}
&(\partial_t\partial_t^\alpha v)(t)
= \kernel^\alpha(t)v_t(0)+\partial_t^\alpha v_t(t)
\end{aligned}
\]
as well as  
\[
\begin{aligned}
&\partial_t^{\altil}\left(f(\uinf+\usim)-f(\uinf)-f'(\uinf)\usim\right)\\
&=\kernel^{\altil}*\left(f'(\uinf+\usim)-f'(\uinf)\right)\usim_t
=\kernel^{\altil}*\Bigl(\int_0^1f''(\uinf+\theta\usim)\,d\theta\,\usim\,\usim_t\Bigr).
\end{aligned}
\]
We test with $\mu\partial_t^{\altil}\usim-\mathbb{L}\partial_t^{\altil}\usim$, thus focusing on the higher regularity setting $s=1$.
This, analogously to \eqref{enest_usim}, yields
\begin{equation}\label{enest_usim_t}
\begin{aligned}
&\min\left\{\frac{1+2\|\pinf\|_{L^\infty(\Omega)}^2}{2\underline{c}},\, \underline{c}_D\right\} \partial_t^\alpha\|\partial_t^{\altil}\usim(t)\|_{H^1(\Omega)}^2
+\frac{1}{2(C_{\text{ell}}^\Omega)^2}\, \|\partial_t^{\altil}\usim(t)\|_{H^2(\Omega)}^2
\\
&\leq 
\left(\frac{1+2\|\pinf\|_{L^\infty(\Omega)}^2}{2\underline{c}^2}+2\right)\,
\|p\,\kernel^{\altil}*(\bar{\bar{f}}\,\usim_t)(t)-\partial_t^{\altil} r(t)+\chi_{\alpha<1}\kernel^\alpha(t)u_t(0)\|_{L^2(\Omega)}^2
\end{aligned}
\end{equation}
with the indicator variable $\chi_A=\begin{cases}1&\text{ if $A$ holds}\\0&\text{ else}\end{cases}$ and 
\begin{equation}\label{fbarbar}
\bar{\bar{f}}(t,x) = f'(u(t,x))-f'(\uinf(x))=\int_0^1f''(\uinf(x)+\theta\usim(t,x))\,d\theta\,\usim(t,x).
\end{equation}
In case $\altil<1$, the convolution term on the right hand side of \eqref{enest_usim_t} is problematic, since it contains $\usim_t$, which is a higher derivative than what we can estimate in case $\altil<1$.
In order to tackle it, we would need a commutator estimate. However such an estimate typically takes the form \cite[Corollary 1.4 (2) with $A^{\altil}=\partial_t^{\altil}$]{DongLi2019}  
\begin{equation}\label{estDongLi}
\|\partial_t^{\altil}(h\,g)-h\,\partial_t^{\altil}\|_{L^\mathfrak{p}(0,t)}\leq 
C(\altil,\alpha_1,\mathfrak{p})\|\partial_t^{\alpha_1}h\|_{L^\mathfrak{p}(0,t)} \|\partial_t^{\altil-\alpha_1}g\|_{L^\infty(0,t)}
\end{equation}
for $0\leq \alpha_1\leq \altil$ and the nonlocal character of the right hand side does not allow to reproduce decay as $t\to\infty$.

Decay as $t\to\infty$ of $\partial_t^{\altil}\usim$ can therefore only be rigorously be verified in case $\altil=1$ in which the right hand side in \eqref{enest_usim_t} under a growth condition \eqref{growth} can be estimated similarly to \eqref{estgrowth_m2}
\begin{equation}\label{estgrowth_m2_alpha1}
\begin{aligned}
&\|p\int_0^1f''(\uinf+\theta\usim(t))\,d\theta\,\usim(t)\,\usim_t(t)\|_{L^2(\Omega)}\\
&\leq 
\|p\|_{L^{\rrr}(\Omega)}
\bigl(\Phi_\infty^1(\|\uinf\|_{H^1(\Omega)})+\Phi_\sim^1(\|\usim(t)\|_{H^1(\Omega)})\bigr)\|\usim(t)\|_{H^1(\Omega)}
\,\|\usim_t(t)\|_{H^2(\Omega)},
\end{aligned}
\end{equation}
where through the barrier argument at the beginning of the proof of Theorem~\ref{thm:decay_usim} we have already guaranteed that under a smallness assumption on $\|\usim(0)\|_{H^1(\Omega)}$, the bound
\[\begin{aligned}
\left(\frac{1+2\|\pinf\|_{L^\infty(\Omega)}^2}{2\underline{c}^2}+2\right)\,
\|p\|_{L^{\rrr}(\Omega)}
\bigl(\Phi_\infty^1(\|\uinf\|_{H^1(\Omega)})+\Phi_\sim^1(\|\usim(t)\|_{H^1(\Omega)})\bigr)\|\usim(t)\|_{H^1(\Omega)}\\
\leq\frac{1}{4(C_{\text{ell}}^\Omega)^2}\text{ for all }t>0 
\end{aligned}\]
holds. 
Using this in \eqref{enest_usim_t} implies
\begin{equation}\label{enest_usim_t_hi}
\begin{aligned}
&\partial_t^\alpha\|\usim_t(t)\|_{H^1(\Omega)}^2
+c_0\, \|\usim_t(t)\|_{H^1(\Omega)}^2  
+c_1\, \|\usim_t(t)\|_{H^2(\Omega)}^2\\  
&\leq\tilde{C}\Bigl(\|r_t(t)\|_{L^2(\Omega)}^2+\kernel^\alpha(t)^2\|u_t(0)\|_{L^2(\Omega)}^2\revision{\Bigr)},
\end{aligned}
\end{equation} 
\Margin{RevX 9.}
for some $c_0,\,c_1\,>0$. 
\\
Thus, in case $\alpha=1$ 
\begin{equation}\label{decay_usim_t}
\begin{aligned}
&\|u_t(t)\|_{H^1(\Omega)}^2=\|\usim_t(t)\|_{H^1(\Omega)}^2\\
&\leq e^{-c_0 t} \|\usim_t(0)\|_{H^1(\Omega)}^2
+\tilde{C}\int_0^t e^{-c_0 (t-s)}\,\|r_t(s)\|_{L^2(\Omega)}^2\, ds\quad t>0,
\end{aligned}
\end{equation}
where the right hand side tends to zero if 
\begin{equation}\label{rtL2}
r_t\in L^2(0,\infty;L^2(\Omega)), \quad \|r_t(t)\|_{L^2(\Omega)}\to0\quad \text{ as }t\to\infty
\end{equation} 
due to the estimate
\[
\int_0^t e^{-c_0 (t-s)}\,\|r_t(s)\|_{L^2(\Omega)}^2\, ds
\leq 
 e^{-c_0 t/2}\,\int_0^{t/2}\|r_t(s)\|_{L^2(\Omega)}^2\, ds
+\frac{1-e^{-c_0 t/2}}{c_0}\,\sup_{s\geq t/2}\|r_t(s)\|_{L^2(\Omega)}^2\, ds.
\]
\footnote{the latter is not implied by $r_t\in L^2(0,\infty;L^2(\Omega))$, as the counterexample 
$r_t=\sum_{j\in\mathbb{N}}\chi_{[j,j+j^{-4}]}$ shows.}
thus proving Theorem~\ref{prop:decay_usim_t} (a).

If $\alpha\in (0,1]$ and $\|r_t(t)\|_{L^2(\Omega)}^2$ not only tends to zero but decays at a power law / exponential rate, we can also invoke Lemma~\ref{lem:eta01} with 
$\eta_0(t)=\|u_t(t+t_0)\|_{H^1(\Omega)}$, 
$\eta_1(t)=\|u_t(t+t_0)\|_{H^2(\Omega)}$, 
to obtain, under the a priori assumption $\eta_0(0)=\|u_t(t_0)\|_{H^1(\Omega)}<\infty$ and using the known asymptotics of $\kernel^\alpha(t)^2$ in case $\alpha<1$
\begin{itemize}
\item[(b)] If $T=\infty$ and $\|r_t(t)\|_{L^2(\Omega)}^2\leq C_r t^{-\beta}$:
\[ 
\|u_t(t)\|_{H^2(\Omega)}^2\lesssim \tilde{C}(C_r+\chi_{\alpha<1}\Gamma(1-\alpha)^{-2}\|u_t(0)\|_{L^2(\Omega)}^2)\, (t-t_0)^{-\min\{\alpha,\beta\}}\text{ as }t\to\infty
\]
\item[(c)] If $\alpha=1$ and $\|r_t(t)\|_{L^2(\Omega)}^2\leq C_r e^{-c_1 t}$, $t\in (t_0,T)$:
\[ 
\|u_t(t)\|_{H^2(\Omega)}^2\lesssim \tilde{C}C_r e^{-\tilde{c}_0 (t-t_0)}, \ t\in (t_0,T).
\]
\end{itemize}
Estimate \eqref{powerdecay_eta} in Lemma~\ref{lem:eta} implies 
(c) in case $\alpha\in(0,1)$.

\bigskip

\section{Proof of Theorem~\ref{thm:decay_usim_nosmallness}}\label{sec:proof_decay_noinitialsmallness}
Under condition \eqref{boundpqf_int}, the alternative characterization of the  \revision{difference} 
\Margin{RevX 10.}
$\usim$ between a solution to \eqref{ibvp_usim} and a steady state via the PDE
\begin{equation}\label{PDE_peff}
\partial_t^\alpha u^\sim+(-\mathbb{L}+\peff(t))\usim=r-\rinf
\end{equation}
with the identity and notation
\[
\begin{aligned}
&f(u(t,x))-f(\uinf(x))=\int_0^1 f'(\theta u(t,x)+(1-\theta)\uinf(x))\,d\theta\,(u(t,x)-\uinf(x))\\
&\phantom{f(u(t,x))-f(\uinf(x))}
=:\bar{f}(t,x)\,(u(t,x)-\uinf(x)) \\
&\peff(t,x)=p(x)\bar{f}(t,x)-q(x)
\end{aligned}
\] 
is useful.
Testing \eqref{PDE_peff} with $(\mu-\mathbb{L})\usim$ under condition \eqref{boundpqf_int} leads to 
\begin{equation}\label{enest_usim_int}
\begin{aligned}
&
\min\{\mu,\, \underline{c}_D\}
\partial_t^\alpha\|\usim(t)\|_{H^1(\Omega)}^2
+\frac{\min\{\mu\,\min\{\tfrac{\underline{c}}{2},c_D\}-D(t),\,\frac12\}}{(C_{\text{ell}}^\Omega)^2}\,
\|\usim(t)\|_{H^2(\Omega)}^2
\\
&\leq 
\left(\frac{\mu}{\underline{c}}+2\right)\,
\|r(t)-\rinf\|_{L^2(\Omega)}^2\\
&\text{with }D(t)=C\Bigl(\bigl(\|p\|_{L^{\rrr}(\Omega)}\bigl(\Phi_\infty^1(\|\uinf\|_{H^1(\Omega)})+\Phi_\sim^1(\|\usim(t)\|_{H^1(\Omega)})\bigr)
+\|q\|_{L^2(\Omega)}\bigr)\,C_{H^2\to L^\infty}^\Omega\Bigr)^2
\end{aligned}
\end{equation}
in place of \eqref{enest_usim} with \eqref{estgrowth_m2}, where we have used
\begin{equation}\label{est_peff}
\begin{aligned}
&\|\peff(t)\usim(t)\|_{L^2(\Omega)}
&\leq \Bigl(\|p\|_{L^\rrr(\Omega)} \|\bar{f}(t)\usim(t)\|_{L^{2\rrr/(\rrr-2)}(\Omega)} + \|q\|_{L^2(\Omega)}\Bigr)\|\usim(t)\|_{L^\infty(\Omega)}
\end{aligned}
\end{equation}
(note that $D(t)$ is defined by only using $\|\usim(t)\|_{H^1(\Omega)}$ here in order to be able to make use of a large value of $\mu$ to dominate it) as well as (cf. \eqref{f0fprime0}) 
\[
|f'(\xi)|=|\int_0^\xi f''(\xi)\, d\xi| \leq c_2|\xi|+C_2|\xi|^{\kappa_2+1}
\]
which means that we obtain the same growth estimate for $f'(\xi)$ as for $f''(\xi)\xi$.
We choose $\mu$ large enough so that 
$D(0) < \frac12 \mu\,\min\{\tfrac{\underline{c}}{2},c_D\}$
and use a barrier argument to conclude that 
$D(t) < \frac12 \mu\,\min\{\tfrac{\underline{c}}{2},c_D\}$
remains valid for all $t>0$. 
Lemma~\ref{lem:eta} can thus be applied to yield \eqref{decay_usim} with $s=1$, provided $\|r(t)-\rinf\|_{L^2(\Omega)}^2=O(\decayfunc(t))$. 

Likewise, the time differentiated version with $\altil\in\{\alpha,1\}$
\[
(\partial_t^\alpha -\mathbb{L}+\peff)\partial_t^{\altil}\usim
=\peff\partial_t^{\altil}\usim-\partial_t^{\altil} (\peff\,u^\sim)\,+\partial_t^{\altil} r -\chi_{\alpha<1}\chi_{\altil=1}\kernel^\alpha(t)u_t(0).
\]
with 
\begin{equation}\label{comm_ut_int}
\partial_t^{\altil} (\peff\,u^\sim)-\peff\partial_t^{\altil}\usim
= p(\partial_t^{\altil}(\bar{f}\,u^\sim)-\bar{f}\,\partial_t^{\altil} u^\sim))
\end{equation}
implies
\[
\begin{aligned}
&\min\{\mu,\, \underline{c}_D\} 
\partial_t^\alpha\|\partial_t^{\altil}\usim(t)\|_{H^1(\Omega)}^2
+\frac{\min\{\mu\,\min\{\tfrac{\underline{c}}{2},c_D\},\,1\}}{4(C_{\text{ell}}^\Omega)^2}\, \|\partial_t^{\altil}\usim(t)\|_{H^2(\Omega)}^2
\\
&\leq 
\left(\frac{\mu}{\underline{c}}+2\right)\,
\|-p(\partial_t^{\altil}(\bar{f}\,u^\sim)-\bar{f}\,\partial_t^{\altil} u^\sim))+\partial_t^{\altil} r(t)-\chi_{\alpha<1}\chi_{\altil=1}\kernel^\alpha(t)u_t(0)\|_{L^2(\Omega)}^2,
\end{aligned}
\]
with $\mu$ chosen as above.
Bounding the right hand side requires a commutator estimate on \eqref{comm_ut_int}, which for $\altil<1$ we can only guarantee in the nonlocal sense \eqref{estDongLi}, thus not implying decay.
In case $\altil=1$, under the growth condition \eqref{growth} we get, exactly as in \eqref{estgrowth_m2_alpha1},
\[
\begin{aligned}
&\|-p(\partial_t(\bar{f}\,u^\sim)(t)-\bar{f}(t)\,\partial_t u^\sim(t)))\|_{L^2(\Omega)}\\
&\leq\|p\|_{L^{\rrr}(\Omega)}
\bigl(\Phi_\infty^1(\|\uinf\|_{H^1(\Omega)})+\Phi_\sim^1(\|\usim(t)\|_{H^1(\Omega)})\bigr)\|\usim(t)\|_{H^1(\Omega)}
\,\|\usim_t(t)\|_{H^2(\Omega)},
\end{aligned}
\]
where we can use the previously established decay estimate \eqref{decay_usim} with $s=1$ on $\usim(t)$ to deduce existence of $t_1\geq0$ such that 
\[\begin{aligned}
\left(\frac{\mu}{\underline{c}}+2\right)
\|p\|_{L^{\rrr}(\Omega)}
\bigl(\Phi_\infty^1(\|\uinf\|_{H^1(\Omega)})+\Phi_\sim^1(\|\usim(t)\|_{H^1(\Omega)})\bigr)^2\|\usim(t)\|_{H^1(\Omega)}^2
\\
\leq \frac{\min\{\mu\,\min\{\tfrac{\underline{c}}{2},c_D\},\,1\}}{4(C_{\text{ell}}^\Omega)^2}
\qquad t\geq t_1. 
\end{aligned}\]
From this, Theorem~\ref{thm:decay_usim_nosmallness} follows analogously to 
Theorems~\ref{thm:decay_usim} and \ref{prop:decay_usim_t}.
\section*{Appendix: proof of Lemma~\ref{lem:eta}}
We only provide the somewhat more involved case $\alpha\in(0,1)$ here.

First of all, note that the regularity assumption on $\eta$ imply continuity of $\eta$ and $\partial_t^\alpha\eta$ on $(0,T)$ and therefore the word ``almost'' can be skipped in \eqref{enest_eta}.
Since $\nu$ satisfies \eqref{enest_eta} with equality (cf., e.g. \cite[Theorem 5.4]{BBB}) for all $t\in(0,T)$, as well as $\nu(0)=\eta(0)$, we obtain, for the difference $\eta-\nu$,  
\[
\partial_t^\alpha(\eta-\nu)(t)+c_0(\eta-\nu)(t) \leq 0 \text{ for all }t\in(0,T),
\quad (\eta-\nu)(0)=0
\]
It suffices to conclude from this that $(\eta-\nu)(t) \leq 0$ for almost every $t\in(0,T)$.
To this end, we can proceed exactly as in the proof of (5.20) in \cite{KubicaRyszewskaYamamoto:2020}, see (5.21) there:
Assume contrarily that there exists $t_0\in(0,T]$ such that 
$(\eta-\nu)(t_0) = \max_{t\in[0,T]}(\eta-\nu)(t) > 0$. 
We can now use the maximum principle 
\begin{quote}
\textbf{Lemma} (\cite[Lemma 5.2]{KubicaRyszewskaYamamoto:2020})
Let $f\in W^{1,1}(0,T)$ attain its maximum over the interval $[0,T]$ at a point
$t_0 \in (0,T]$. If for each $\underline{t}\in(0,T)$, there exists $\beta\in(0, 1]$ such that $f\in W^{1,1/(1-\beta)} (\underline{t},T)$, then $(\partial_t^\alpha f )(t_0 ) \geq 0$ for every $\alpha\in(0, \beta)$.
\end{quote}
together with the fact that $\eta-\nu\in W^{1,\infty}(\underline{t},T)$ for any $\underline{t}\in(0,T)$) to conclude that $\partial_t^\alpha (\eta-\nu) (t_0 ) \geq 0$ and hence altogether 
$\partial_t^\alpha(\eta-\nu)(t_0)+c_0(\eta-\nu)(t_0) > 0$, a contradiction.\\
This proves \eqref{eta_leq_nu}.

In case $\phi(t)=\phi_0 t^{-\alpha}$ we can estimate similarly to the proof of \cite[Lemma 10.11]{BBB} make use of the asympotics of Mittag Leffler functions and their interrelations, cf., e.g., \cite[Section 3.4]{BBB}, more precisely
\[
\begin{aligned}
&c_0 s^{\alpha-1}E_{\alpha,\alpha}(-c_0 s^\alpha)=-\frac{d}{ds} E_{\alpha,1}(-c_0 s^\alpha)\\
&0\leq E_{\alpha,1}(-x)\leq E_{\alpha,1}(0)=1\\ 
&E_{\alpha,1}(-x)\leq \frac{1}{1+\Gamma(1+\alpha)^{-1}x}, \quad
E_{\alpha,\alpha}(-x)\leq C_\alpha x^{-1}
\end{aligned}
\]
that hold for every $\alpha\in(0,1)$, 
all $s\in\mathbb{R}$ and all $x\in\mathbb{R}^+$,
as well as the complete monotonicity of the function $x\mapsto E_{\alpha,1}(-x)$ on $\mathbb{R}^+$. 
This allows us to estimate as follows 
\[
\begin{aligned}
\int_0^t&(t-s)^{\alpha-1}E_{\alpha,\alpha}(-c_0 (t-s)^\alpha) \,s^{-\alpha}\,ds\\
&= \int_0^{\frac{t}{2}} s^{\alpha-1}E_{\alpha,\alpha}(-c_0 s^\alpha)  (t-s)^{-\alpha}\,ds 
+\int_{\frac{t}{2}}^t s^{\alpha-1}E_{\alpha,\alpha}(-c_0 s^\alpha)  (t-s)^{-\alpha}\,ds
\\
&= \frac{\alpha}{c_0}\int_0^{ \frac{t}{2}} E_{\alpha,1}(-c_0 s^\alpha)  (t-s)^{-\alpha-1}\,ds 
-  \frac{\alpha}{c_0}t^{-\alpha}
+ \frac{\alpha}{c_0}E_{\alpha,1}\left(-c_0 \left(\frac{t}{2}\right)^\alpha\right)  \left(\frac{t}{2}\right)^{-\alpha}\\
&\qquad\qquad\qquad + 
\int_{ \frac{1}{2}}^1  \sigma^{\alpha-1} (1\!-\!\sigma)^{-\alpha}\,E_{\alpha,\alpha}(-c_0 t^\alpha \sigma^\alpha)\,d\sigma
\\
&\leq \frac{\alpha}{c_0}\int_0^{ \frac{t}{2}} (t\!-\!s)^{-\alpha-1}\,ds 
+ \frac{\alpha}{c_0} \left(\frac{t}{2}\right)^{-\alpha} +  \frac{C_\alpha}{c_0} \left(\frac{t}{2}\right)^{-\alpha}\int_{\frac{1}{2}}^1  \sigma^{\alpha-1} (1\!-\!\sigma)^{-\alpha}\,d\sigma\\
&\leq \frac{C(\alpha)}{c_0} \, t^{-\alpha}\qquad\text{ with }
C(\alpha)=2^\alpha\,\left(1+\alpha+C_\alpha\int_{\frac{1}{2}}^1  \sigma^{\alpha-1} (1\!-\!\sigma)^{-\alpha}\,d\sigma\right)\,,
\end{aligned}
\]
where we have used the substitution $s=t\sigma$.

In case of nonnegative $\eta$, we make use of 
\[
\begin{aligned}
&(\kernel^{1-\alpha}*\partial_t^\alpha v)(t)=v(t) \text{ for }v\in W^{1,1}(0,T)\text{ with }v(0)=0\\
&(\kernel^{1-\alpha}*v)(t)\geq0 \text{ for }v\in L^1(0,T)\text{ with }v(t)\geq0 \text{ for almost every }t\in(0,T)
\end{aligned}
\]
and convolve both sides of 
\[
\partial_t^\alpha(\eta(t)-\eta(0))+c_0\eta(t) \leq \phi(t) \text{ for almost every }t\in(0,T),
\]
that follows directly from \eqref{enest_eta}, with $\kernel^{1-\alpha}$ to arrive at 
\begin{equation}\label{etak1malpha}
\eta(t)-\eta(0) + c_0 \, (\kernel^{1-\alpha}*\eta)(t)\leq (\kernel^{1-\alpha}*\phi)(t)
\end{equation} 
Due to nonnegativity of the $c_0$ term, this yields \eqref{bound_eta_phi}.
\\
To conclude \eqref{bound_eta} from \eqref{bound_eta_phi} under the imposed decay condition of $\phi$, we 
do a substitution $s=t\sigma$ as above to obtain
\begin{equation}\label{estconst}
\begin{aligned}
&\int_0^t (t-s)^{\alpha-1}s^{-\alpha} \,ds=\int_0^1 (1-\sigma)^{\alpha-1}\sigma^{-\alpha} \,d\sigma \\
&\leq \int_0^{\frac12} (1-\sigma)^{\alpha-1}\sigma^{-\alpha} \,d\sigma + \int_{\frac12}^1 (1-\sigma)^{\alpha-1}\sigma^{-\alpha} \,d\sigma\\
&\leq 2^{1-\alpha} \int_0^{\frac12} \sigma^{-\alpha} \,d\sigma +2^{\alpha}\int_{\frac12}^1 (1-\sigma)^{\alpha-1} \,d\sigma
=\frac{1}{1-\alpha}+\frac{1}{\alpha}.
\end{aligned}
\end{equation} 

The bound \eqref{bound_eta_0} with $\phi(t)\equiv \phi_0$ follows from
$E_{\alpha,1}(-c_0t^\alpha)\leq E_{\alpha,1}(0)=1$ and $\int_0^t (t-s)^{\alpha-1}E_{\alpha,\alpha}(-c_0(t-s)^\alpha)\, ds=\frac{1}{c_0}(E_{\alpha,1}(0)-E_{\alpha,1}(-c_0t^\alpha))\leq\frac{1}{c_0}$.

\section*{Acknowledgments}
This research was funded in part by the Austrian Science Fund (FWF) 
[10.55776/F100800].
The author thanks the anonymous reviewers for their careful reading of the manuscript and their valuable comments.










\begin{thebibliography}{10}

\bibitem{AkagiSchimpernaSegatti:2016}
Goro Akagi, Giulio Schimperna, and Antonio Segatti.
\newblock Fractional {Cahn}-{Hilliard}, {Allen}-{Cahn} and porous medium
  equations.
\newblock {\em J. Differ. Equations}, 261(6):2935--2985, 2016.

\bibitem{Alikhanov:10}
AA~Alikhanov.
\newblock A priori estimates for solutions of boundary value problems for
  fractional-order equations.
\newblock {\em Differential equations}, 46(5):660--666, 2010.

\bibitem{AllenCahn:1972}
S.M Allen and J.W Cahn.
\newblock Ground state structures in ordered binary alloys with second neighbor
  interactions.
\newblock {\em Acta Metallurgica}, 20(3):423--433, 1972.

\bibitem{AronsonWeinberger}
{Donald G.} Aronson and {Hans F} Weinberger.
\newblock Multidimensional nonlinear diffusion arising in population genetics.
\newblock {\em Advances in Mathematics}, 30(1):33--76, October 1978.

\bibitem{BerestyckiNadinPerthameRyzhik:2009}
Henri Berestycki, Grégoire Nadin, Benoit Perthame, and Lenya Ryzhik.
\newblock The non-local {Fisher-KPP} equation: travelling waves and steady
  states.
\newblock {\em Nonlinearity}, 22(12):2813, oct 2009.

\bibitem{CabreRoquejoffre:2012}
Xavier Cabr\'{e} and Jean-Michel Roquejoffre.
\newblock The influence of fractional diffusion in {Fisher-KPP} equations.
\newblock {\em Communications in Mathematical Physics}, 320, 02 2012.

\bibitem{DuYangZhou:2020}
Qiang Du, Jiang Yang, and Zhi Zhou.
\newblock Time-fractional {Allen}-{Cahn} equations: analysis and numerical
  methods.
\newblock {\em J. Sci. Comput.}, 85(2):29, 2020.
\newblock Id/No 42.

\bibitem{Fisher:1937}
Ronald~A. Fisher.
\newblock The wave of advance of advantageous genes.
\newblock {\em Annals of Eugenics}, 7(4):355–369, 1937.

\bibitem{FloridiaLiuYamamoto:2023}
Giuseppe Floridia, Yikan Liu, and Masahiro Yamamoto.
\newblock Blowup in {{\(L^1(\Omega )\)}}-norm and global existence for
  time-fractional diffusion equations with polynomial semilinear terms.
\newblock {\em Adv. Nonlinear Anal.}, 12:15, 2023.
\newblock Id/No 20230121.

\bibitem{GilbargTrudinger}
David Gilbarg and Neil~S. Trudinger.
\newblock {\em Elliptic partial differential equations of second order}.
\newblock Class. Math. Berlin: Springer, reprint of the 1998 ed. edition, 2001.

\bibitem{Jin:2021}
Bangti Jin.
\newblock {\em Fractional Differential Equations: An Approach via Fractional
  Derivatives}.
\newblock Springer, 2021.

\bibitem{BB2}
Barbara Kaltenbacher and William Rundell.
\newblock On an inverse potential problem for a fractional reaction-diffusion
  equation.
\newblock {\em Inverse Problems}, 35:065004, 2019.
\newblock open access:
  https://iopscience.iop.org/article/10.1088/1361-6420/ab109e.

\bibitem{BBB}
Barbara Kaltenbacher and William Rundell.
\newblock {\em Inverse Problems for Fractional Partial Differential Equations}.
\newblock Number 230 in Graduate Studies in Mathematics. American Mathematical
  Society, 2023.

\bibitem{BB19}
Barbara Kaltenbacher and William Rundell.
\newblock Identification of space-dependent coefficients in two competing terms
  of a nonlinear subdiffusion equation.
\newblock {\em Inverse Problems}, 2026.
\newblock to appear; see also arXiv:2601.21018 [math.NA].

\bibitem{KPP}
A.~Kolmogorov, I.~Petrovskii, and N.~Piskunov.
\newblock A study of the diffusion equation with increase in the amount of
  substance, and its application to a biological problem.
\newblock {\em Moscow Univ., Math. Mech.}, 1:1--25, 1937.

\bibitem{KubicaRyszewskaYamamoto:2020}
Adam Kubica, Katarzyna Ryszewska, and Masahiro Yamamoto.
\newblock {\em Time-fractional Differential Equations: A Theoretical
  Introduction}.
\newblock Springer, 2020.

\bibitem{DongLi2019}
Dong Li.
\newblock On {Kato}-{Ponce} and fractional {Leibniz}.
\newblock {\em Rev. Mat. Iberoam.}, 35(1):23--100, 2019.

\bibitem{LuchkoYamamoto:2025}
Yuri Luchko and Masahiro Yamamoto.
\newblock Comparison principles for the time-fractional diffusion equations
  with the {Robin} boundary conditions. {II}: {Semilinear} equations.
\newblock {\em Fract. Calc. Appl. Anal.}, 28(5):2198--2240, 2025.

\bibitem{NinomiyaHirokazuTaniguchi:2005}
Hirokazu Ninomiya and Masaharu Taniguchi.
\newblock Existence and global stability of traveling curved fronts in the
  {Allen}-{Cahn} equations.
\newblock {\em J. Differ. Equations}, 213(1):204--233, 2005.

\bibitem{PenroseFife:1990}
Oliver Penrose and Paul~C. Fife.
\newblock Thermodynamically consistent models of phase-field type for the
  kinetic of phase transitions.
\newblock {\em Physica D: Nonlinear Phenomena}, 43(1):44--62, 1990.

\bibitem{QuittnerSouplet:2019}
Pavol Quittner and Philippe Souplet.
\newblock {\em Superlinear parabolic problems. {Blow}-up, global existence and
  steady states}.
\newblock Birkh{\"a}user Adv. Texts, Basler Lehrb{\"u}ch. Cham: Birkh{\"a}user,
  2nd revised and updated edition edition, 2019.

\bibitem{Troeltzsch2010}
Fredi Tr{\"o}ltzsch.
\newblock {\em Optimal Control of Partial Differential Equations: Theory,
  Methods, and Applications}.
\newblock Number 112 in Graduate Studies in Mathematics. American Mathematical
  Society, 2010.

\bibitem{FrankZeldovich}
Ya~B. Zeldovich and D.~A. Frank-Kamenetskii.
\newblock Theory of thermal flame propagation {(in Russian)}.
\newblock {\em J Phys Chem}, 12:100--105, 1938.

\end{thebibliography}
\end{document}